\documentclass{article}
\usepackage{amsfonts,amsmath,amsthm,amscd,amssymb,latexsym}
\usepackage{graphicx}
\DeclareGraphicsExtensions{.eps}

\begin{document}
\title{Examples of defining groups by finite automata}
\author{Victoriia Korchemna}
\thanks {The author expresses her thanks to A. S. Oliynyk for introducing her to the topic of automata transformations} 
\date{28.10.2018}
\theoremstyle{plain}
\newtheorem{theorem}{Theorem}
\newtheorem{proposition}{Proposition}
\newtheorem{definition}{Definition}
\newtheorem{remark}{Remark}
\maketitle 
	
\begin{abstract}
We construct the groups $<A,B,C \;| \; A^2,B^2,C^2,(ABC)^2>$ and\\ $<A,B \;| \; A^2,B^4,(AB)^4>$, using 3-state automata over the alphabets $\{1,2,3\}$ and $\{1,2,3,4\}$. In addition, we show, how to define direct powers of $G$  by automaton (when for $G$ it's given), keeping the alphabet.\\
\end{abstract}	
  
\section{Introduction and definitions}
 \subsection{Rooted tree of words and it's isomorphisms}
 Let $X$ be a finite set, which will be called alphabet with elements called letters. We always suppose $|X|>1$. Let $X^*$ be the free monoid generated by $X$. The elements of this monoid are finite words $x_1x_2 ... x_n,\; x_i \in X$, including the empty word $\emptyset$. Denote by $X^w$ the set of all infinite words $x_1x_2 ... x_n ...,\; x_i \in X$.\\
 The set $X^*$ is naturally a vertex set of a rooted tree, in which two words are connected by an edge if and only if they are of the form $v$ and $vx$, where $v\in X^* $, $x \in X.$ The empty word $\emptyset$ is the root of the tree $X^*$.

 A map $f:X\to X$ is an endomorphism of the tree $X$, if for any two adjacent vertices $v$, $vx\in X^*$ the vertices $f(v)$ and $f(vx)$ are also adjacent, so that there exist $u \in X^*$ and $y \in X$ such that $f(v) = u$ and $f(vx) = uy$. An automorphism is a bijective endomorphism.
 
 \subsection{Restrictions} 
 Let $g : X^* \to X^*$ be an endomorphism of the rooted tree $X$. For each vertex $v \in X^*$ we can determine an endomorphism $g|_v : X^* \to  X^*$ by the condition\\
 $g(vw) = g(v)g|_{v(w)}.$\\
 We call the endomorphism $g|_v$ restriction of $g$ in $v$. We have the following obvious properties of restrictions:\\
 $g|_{v_1v_2} = (g|_{v_1})|_{v_2}$,\\
 $(g_1 \cdot g_2)|_v = g_1|_v \cdot g_2|_{g_1(v)}$.\\\\

 \subsection{Automata}
 An automaton A is a quadruple $(X,Q,\pi,\lambda)$, where
 \begin{itemize}
 	\item $X$ is an alphabet
 	\item $Q$ is a set of states of the automaton
 	\item $\pi : Q \times X \to X$ is a map, called the transition function of the automaton
 	\item $\lambda : Q \times X \to X$ is a map, called the output function of the automaton
 \end{itemize}
 An automaton is finite if it has a finite number of states. The maps $\pi,\lambda$ can be extended on $Q \times X^*$ by the following recurrent formulas:\\
 $\pi(q,\emptyset) = q$, $\pi(q,xw) = \pi(\pi(q,x),w)$,\\
 $\lambda(q,\emptyset) = \emptyset$,  $\lambda(q,xw) = \lambda(q,x)\lambda(\pi(q,x),w)$,\\
 where $x \in X$, $q \in Q$, and $w \in X^*$ are arbitrary elements. Similarly, the maps $\pi,\lambda$ are
 extended on $Q \times X^w$.
 
 An automaton $A$ with a fixed state $q$ is called initial and is denoted by $A_q$.
 Every initial automaton defines the automorphism $\lambda(q,\cdot)$ of the rooted tree $X^*$, which we also denote by $A_q(\cdot) = \lambda(q,\cdot)$ (or q($\cdot$) if it is clear, which automaton it belongs to). We denote by $e$ a trivial state of automaton, i.e., such a state that defines a trivial automorphism of $X^*$. The action of
 an initial automaton $A_q$ can be interpret as the work of a machine, which being in the state $q$ and reading on the input tape a letter $x$, goes to the state $\pi(q,x)$, types on the output tape the letter $\lambda(q,x)$, then moves both tapes to the next position and proceeds further.
 \begin{center} 	
 	\includegraphics[width=8cm, trim=0 450 0 300,clip]{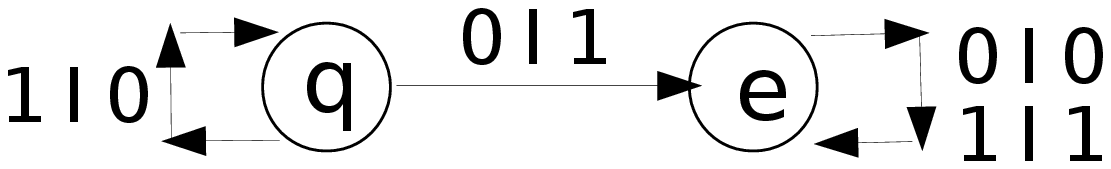} 
 \end{center}
 An automaton $A$ can be represented (and defined) by a labelled directed graph,
 called the Moore diagram, in which the vertices are the states of the automaton and
 for every pair $(q,x) \in Q \times X $ there is an edge from $q$ to $\lambda(q,x)$ labelled by $x|\pi(q,x)$.\\
 Here is the Moore diagram of automaton, called the adding machine. 
 Consider a word of length $l \in \mathbb N$ as a binary number, with lower digits on the left side. If the automaton gets the word in state q, it adds 1 modulo $2^l$ to it. 
  
 \subsection{Permutational wreath products}
 Let H be a group acting (from the left) by permutations on a set X and let G be an arbitrary group. Then the (permutational) wreath product $H \wr G$ is the semi-direct product $H \ltimes G^X$, where $H$ acts on the direct power $G^X$ by the respective permutations of the direct factors.
 Every element of the wreath product $H \wr G$ can be written in the form $h \cdot g$, where $h\in H$ and $g \in G^X$. If we fix some indexing $\{x_1,...,x_d\}$ of the set $X$, then $g$ can be written as $(g_1,...,g_d)$ for $g_i \in G$. Here $g_i$
 is the coordinate of $g$, corresponding to $x_i$. Then multiplication rule for elements $h \cdot (g_1, ... , g_d) \in H \wr G$ is given by the formula:\\
 $\alpha(g_1,...,g_d) \cdot \beta(f_1,...,f_d) = \alpha\beta(g_1f_{\alpha(1)}, ... , g_df_{\alpha(d)})$,
 where $g_i, f_i \in G$, $\alpha, \beta \in H$ and $\alpha(i)$ is the image of $i$ under the action of $\alpha$, i.e., such an index that $\alpha(x_i) = x_{\alpha(i)}$.
 
 \subsection{Wreath recursion}
 We have the following well known fact:
 
 \begin{proposition}
 Denote by $Aut X^*$ the automorphism group of the rooted tree $X^*$ and by $S(X)$ the symmetric group of all permutations of $X$. Fix some indexing $\{x_1, . . . , x_d\}$ of $X$. Then we have an isomorphism\\
 $\psi : AutX^* \to S(X) \wr AutX^*$, given by\\
 $\psi(g) = \alpha(g|_{x_1},...,g|_{x_d})$, where $\alpha$ is the permutation equal to the action of $g$ on $X \subset X^*$.
\end{proposition}

 We usually identify $g \in Aut X^*$ with it's image $\psi(g) \in S(X) \wr Aut X^*$, so that we write $g = \alpha(g|_{x_1},...,g|_{x_d})$. 
 The relation is called wreath recursion. It is a compact way to define recursively
 automorphisms of the rooted tree X. For example, the relation
 $q = \pi(e, q)$,
 where $\pi$ is the transposition $(0, 1)$ of the alphabet $X = \{0, 1\}$, defines an automorphism of the tree $\{0, 1\}^*$, coinciding with the transformation, defined by the state $q$ of the adding machine.
 In general, every invertible finite automaton with the set of states $g_1, ... , g_n$ is described by recurrent formulas:
 \begin{center}
 $g_1 = \tau_1(h_{11}, h_{12}, . . . , h_{1d})$\\
 ...\\
 $g_n = \tau_n(h_{n1}, h_{n2}, . . . , h_{nd})$
 \end{center}
 where $\{h_{ij}:\;1\le i \le n, \; 1\le j \le d\}$ = $\{g_i:\; 1\le i \le n\}$ and $\tau_i$ is the action of $g_i$ on $X$.
 Conversely, any set of formulas of this type, for which $\tau_i$ are arbitrary permutations and each $h_{ij}$ belongs to the set $\{g_1, . . . , g_n\}$, uniquely defines an invertible automaton with the set of states $\{g_1, . . . , g_n\}$.
\\\\
\newpage
\section{Automaton, that defines the group\\ $G_{ABC}=<A,B,C\;|\;A^2,B^2,C^2,(ABC)^2>$}

Although it is known about some groups, that they can be defined by finite automata, such representations are often complicated. For example, number of states of the automata can be much greater then the number of group's generators. From this point of view $G_{ABC}$(one can see it's Kelly's graph below) is a very good group.

\begin{center}		
\includegraphics[width=6cm, trim=100 300 100 100, clip]{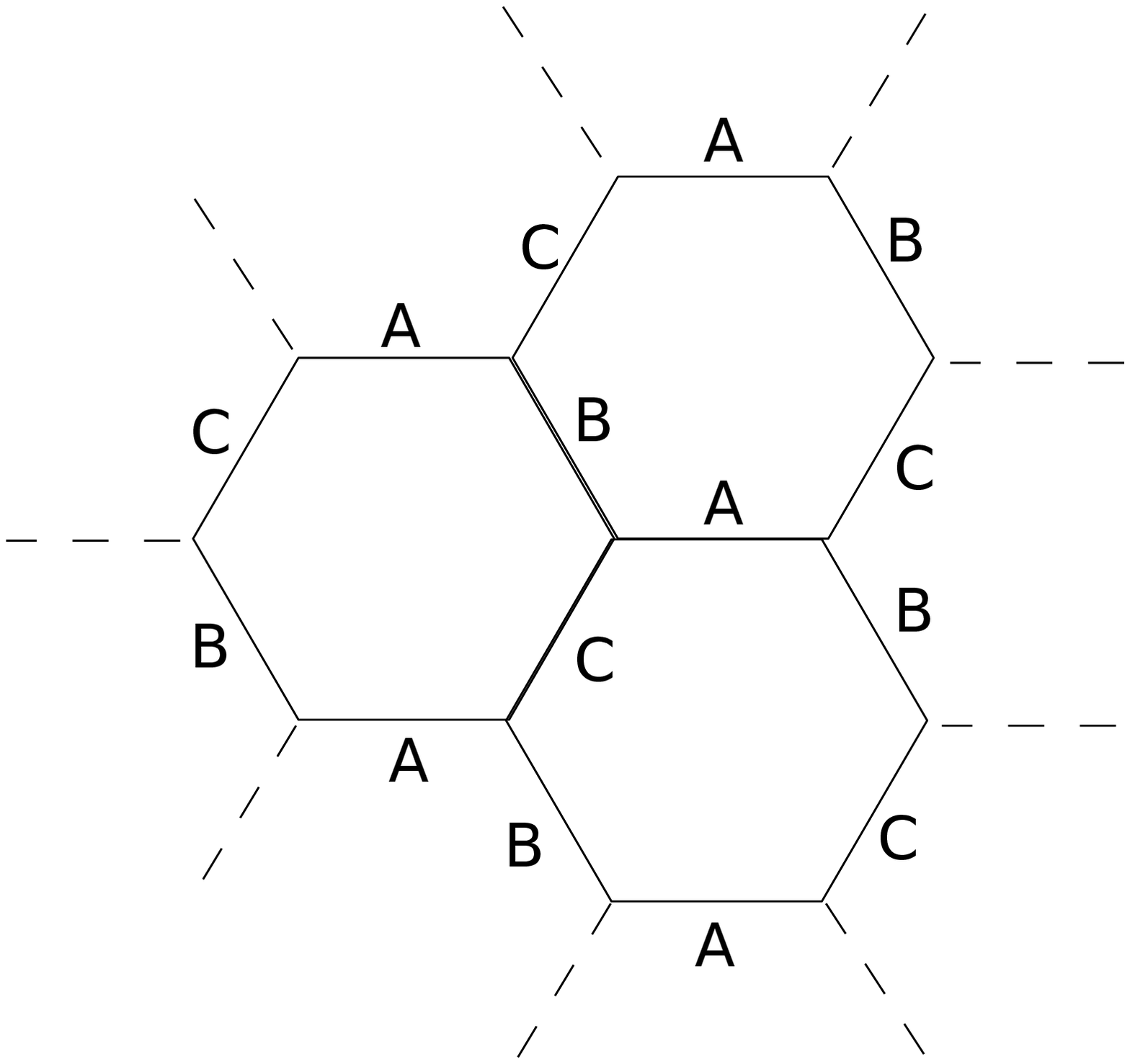}
\end{center}

\begin{theorem}
 $G_{ABC}$ can be defined by automaton with 3 active states. Moreover, these states correspond to generators $A,B,C$ of the group.
\end{theorem}

\begin{proof}
Here is the representation of such automaton by wreath recursion and it's Moor's diagram:\\ 
$a=(a,c,b)$     (1)\\
$b=(c,a,b)$     (2)\\
$c=(12)(e,e,c)$ (3)

\begin{center}	
\includegraphics[width=8cm, trim=50 400 100 200,clip]{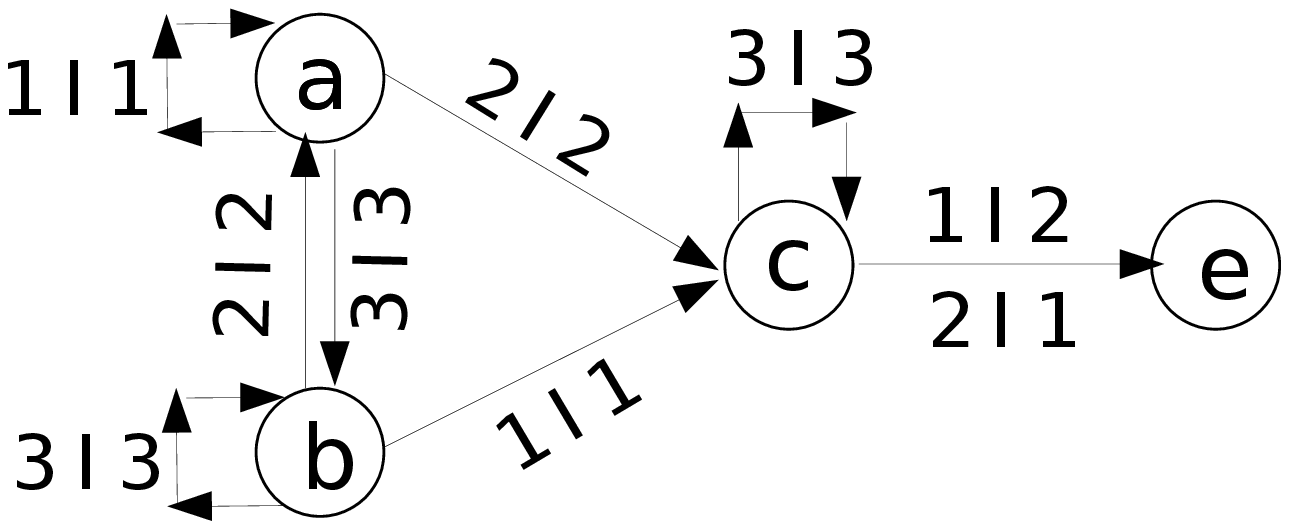}
\end{center}
 
(3) implies $c^2=(e,e,c^2)$, therefore, $c^2=e$.
From (1) and (2) we have:\\
$a^2=(a^2,c^2,b^2)=(a^2,e,b^2)$\\
$b^2=(c^2,a^2,b^2)=(e,a^2,b^2)$\\
Therefore, $a^2=b^2=e$.
Using (1)-(3), we get:\\
$ab=(ac,ca,b^2)=(ac,ca,e)$\\
$abc=(12)(ac,ca,c)$\\
${(abc)}^2=(acca,acca,c^2)=(e,e,e)=e$\\
We are going to show, that the group $<a,b,c>$ doesn't have any extra relations. It's sufficient to prove, that after taking a word of $A, B, C$, which corresponds to not trivial element of $G_{ABC}$, and replacing $A\to a$, $B\to b$, $C \to c$, we get not trivial element of $<a,b,c>$.Redraw the Kelly's graph of $G_{ABC}$:
\begin{center}
\includegraphics[width=10cm, trim=0 425 100 200,clip]{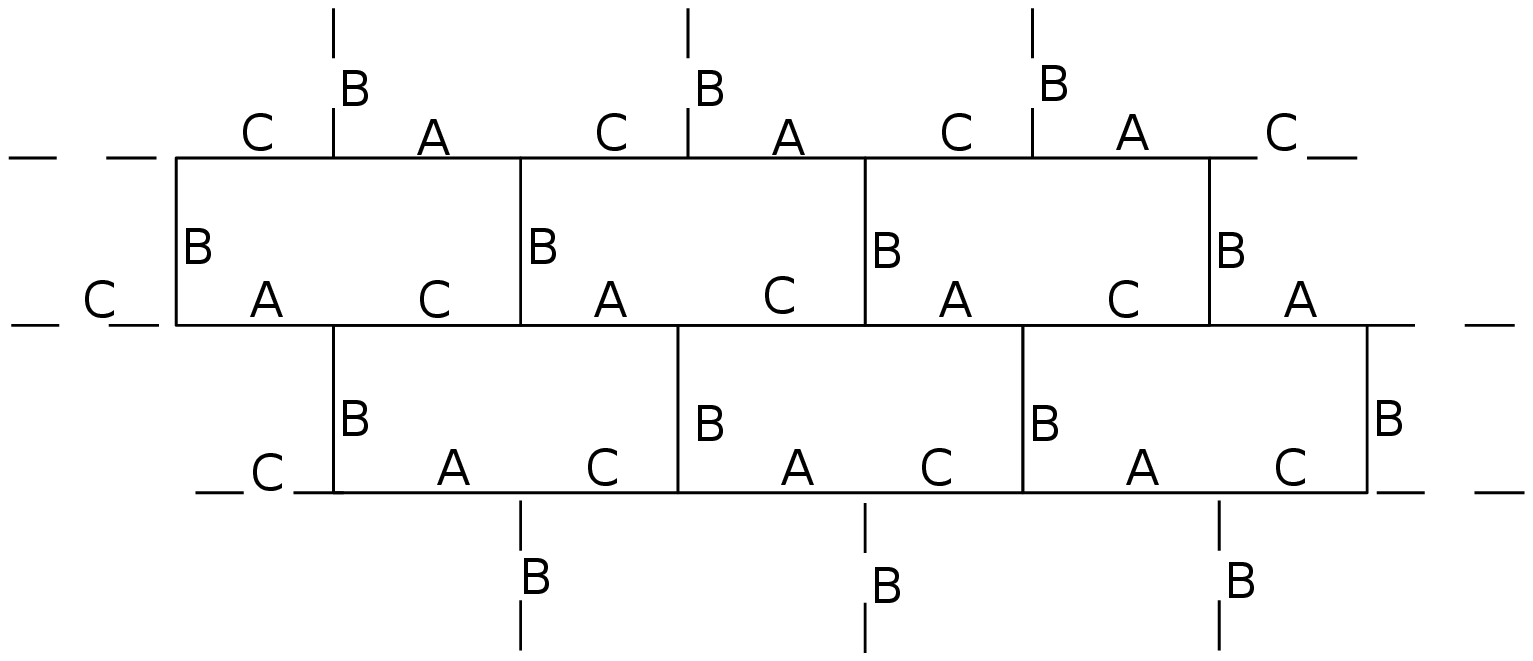}
\end{center}
Assume, to the contrary, that there are paths without cycles, which after replacing $A\to a$, $B\to b$, $C \to c$ define trivial word in $<a,b,c>$.

From the Kelly's graph, we can see that such path can be chosen consisting of 2 parts:\\
1)First part of the path contains $a$ and $b$ only (vertical moving along the graph);\\
2)Second part contains $a$ and $c$ only (horizontal moving).\\
Therefore, it's sufficient to show that all types of words below are not trivial in $<a,b,c>$:\\\\
$[1]:\;{(ab)}^k{(ac)}^m$\\
$[2]:\;{(ab)}^k{(ca)}^m$\\
$[3]:\;{(ab)}^k{(ac)}^ma$\\
$[4]:\;{(ab)}^k{(ca)}^mc$\\
$[5]:\;{b(ab)}^k{(ac)}^m$\\
$[6]:\;{b(ab)}^k{(ca)}^m$\\
$[7]:\;{b(ab)}^k{(ac)}^ma$\\
$[8]:\;{b(ab)}^k{(ca)}^mc$, where $k$ and $m$ are non-negative integers.\\\\
Multiplying [7] left and right by $a$, we get [1]. In the same way [5],[6],[8] can be reduced to [2],[3],[4], so it's enough to consider [1]-[4].

Firstly we will show that all words ${(ab)}^n$, ${(ac)}^n$, ${(bc)}^n$ (and then their inverses ${(ba)}^n$, ${(ca)}^n$, ${(cb)}^n$) are not trivial for every $n \in \mathbb{N}$.\\
All odd powers of $ac$ and $bc$ are not trivial, because their permutations are $(12)$\\
$(ac)^{2k-1} \ne e$, $(bc)^{2k-1} \ne e$ $k \in \mathbb{N}\;$ (*)\\
$ab=(ac,ca,e)\;=>\;(ab)^n=((ac)^n,(ca)^n,e) \;(**)$\\
$bc=(12)(c,a,bc)\;=>\;(bc)^2=(ca,ac,(bc)^2)\;=>$\\
$=>\;(bc)^{2k}=((ca)^k,(ac)^k, (bc)^{2k}) \;(***)$\\
$ac=(12)(a,c,bc)\;=>\;(ac)^2=(ac,ca,(bc)^2)\;=>$\\
$=>\;(ac)^{2k}=((ac)^k,(ca)^k, (bc)^{2k})$\\
 By induction, using non-triviality of odd powers of $(ac)$, we get non-triviality of all it's positive integer powers. Then (**) implies $(ab)^n \ne e$, $n \in \mathbb{N}$ and (***) implies $(bc)^{2k} \ne e$, $k \in \mathbb{N}$. Combining the last result with $(*)$, we get $(bc)^{n} \ne e$, $n \in \mathbb{N}$

Further, if some information (permutation or word) is not needed to make a conclusion, we will replace it by '!'\\
 $(ac)^n=!(!,!,(bc)^n)$\\ 
 $(ca)^n=!(!,!,(cb)^n)$\\
Using obtained representations of  ${(ab)}^n$, ${(ac)}^n$, ${(bc)}^n$, ${(ba)}^n$, ${(ca)}^n$, ${(cb)}^n$, we show non-triviality of [1]-[4]. We consider only cases $k \ge 1, \;m \ge 1$, because $k=0$ and $m=0$ are already considered.\\
$(ab)^k(ac)^m=!(!,!,(bc)^m)\ne e$;\\
$(ab)^k(ca)^m=!(!,!,(cb)^m)\ne e$;\\
$(ab)^k(ac)^ma=!(!,!,(bc)^mb)\ne e$, in another case $(bc)^mb=e$, multiplying the last equality left and right by b, then by c and so on we get $c=e$.\\
Similarly, $(ab)^k(ca)^mc=!(!,!,(cb)^mc)\ne e$\\
It finishes the proof of non-triviality of [1]-[8].
\end{proof}
\newpage 

\section{Automaton, that defines the group\\ $G_{AB}=<A,B \; |\; A^2, B^4, (AB)^4>$}
   
Another group we are going to construct is $G_{AB}$. One can see it's Kelly's graph below. Although $G_{AB}$ has two generators, we will use an automaton with $3$ states: two of then correspond to $A$ and $B$, the third - to $B^2$.

\begin{center}	
\includegraphics[width=10cm, viewport=0 100 700 450, clip]{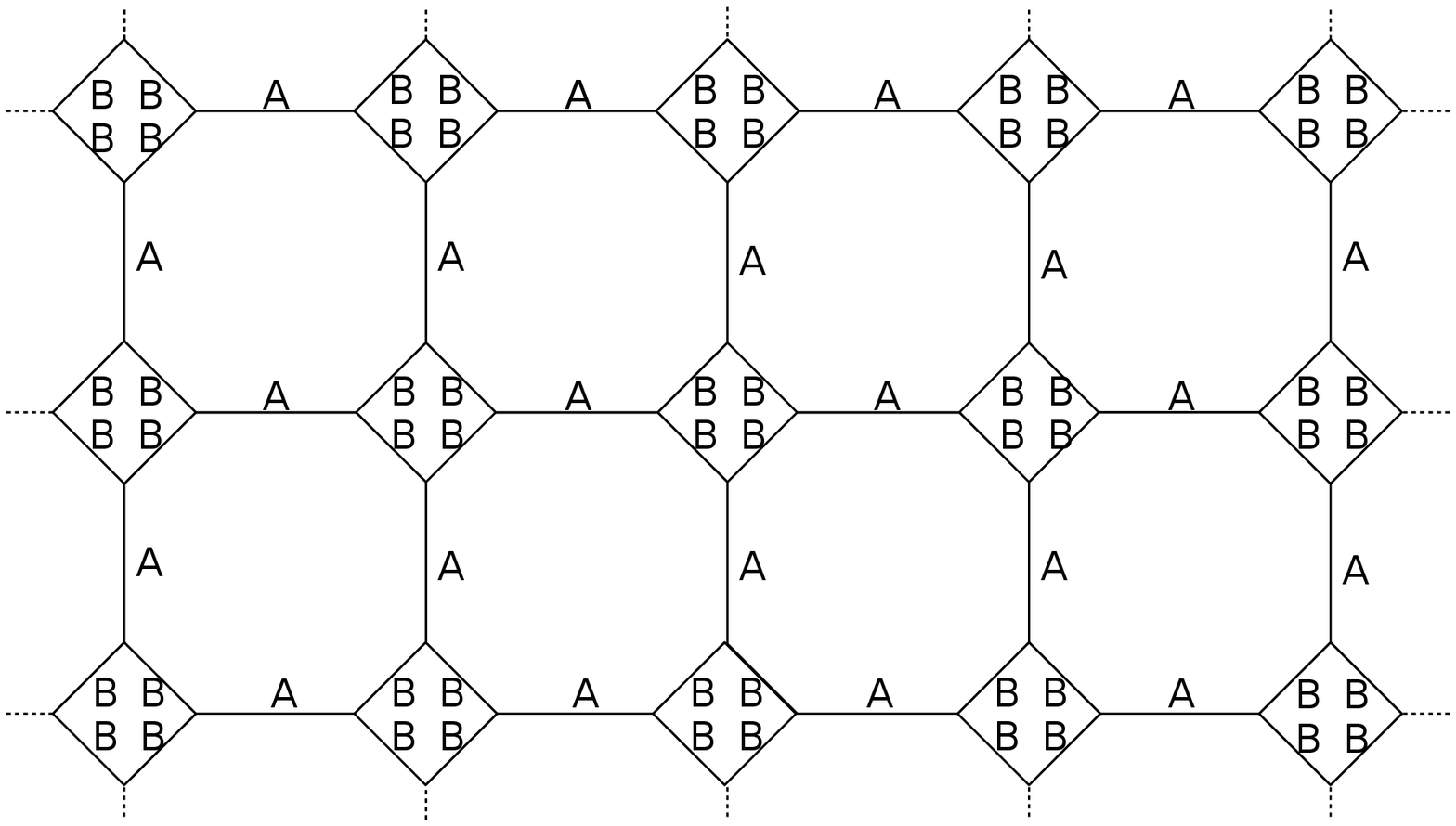} 
\end{center}
	
\begin{theorem}	
 $G_{AB}$ can be defined by automaton with $3$ active states over the alphabet $\{1,2,3,4\}$.
\end{theorem}	

\begin{proof}	

Such automaton can be defined as follows:\\
$a=(c,a,c,a)$\\
$b=(1324)(e,a,e,a)$\\
$c=(12)(34)(e,e,a,a)$\\
We have:\\
$a^2=(c^2,a^2,c^2,a^2)$\\
$c^2=(e,e,a^2,a^2)$

\begin{center}
\includegraphics[width=8cm, trim=75 250 100 350,clip]{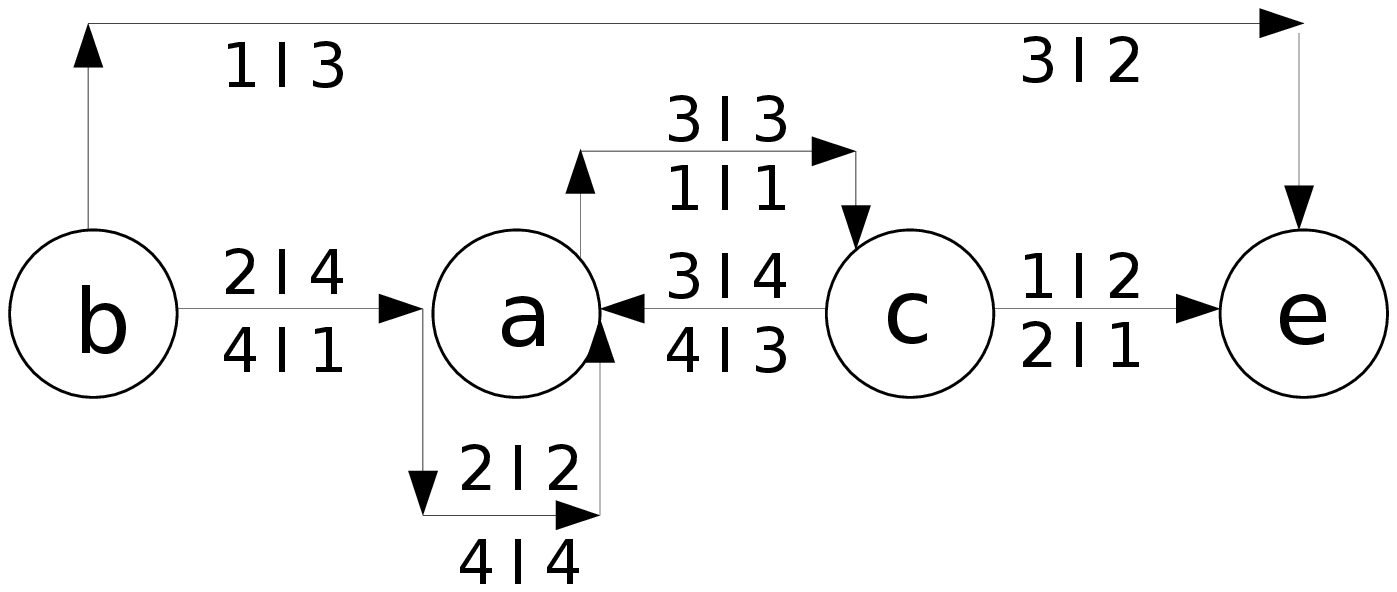} 
\end{center}

The last two equalities imply that $a^2=c^2=e$. Then:\\
$b^2=(12)(34)(e,a^2,a,a)=(12)(34)(e,e,a,a)=c$, so $b^4=c^2=e$ and $a=(b^2, a, b^2, a)$. Further we will use $a$ and $b$ only.\\
$ab=(1324)(b^2,e,b^2,e)$\\
$(ab)^2=(12)(34)(b^4, e, b^2, b^2)=(12)(34)(e, e, b^2, b^2)$\\
$(ab)^4=(e,e,b^4,b^4)=e$.\\
So we have all needed relations in $<a,b>$: $a^2=b^4=(ab)^4=e$. Let us show, that there are no extra relations. It's sufficient to prove, that taking an arbitrary path connecting different vertices of the Kelly's graph of $G_{AB}$ and replacing $A \to a$, $B \to b$, we get a non-trivial element of $<a,b>$. Notice that for every pair of vertices we can choose a path consisting of 5 parts (some of them can be omitted):\\
1)change direction if it is needed (using $B^2$);\\
2)move along the horizontal(vertical) $A$ edges, using $AB^2$ for many times;\\ 
3)rotate (using $B$ or $B^3$);\\
4)move along the vertical(horizontal) $A$ edges, using $AB^2$ for many times;\\ 
5)rotate if it is needed (using $B$, $B^2$ or $B^3$).\\ 
Assume that one of $ab$-words, obtained this way, is trivial. If it starts from $b^2$, we conjugating it by $b^2$ and get another trivial word. So, it's sufficient to show non-triviality of following types of words for each $m \ge 0, n \ge 0$:\\\\
$[1]:\;(ab^2)^n$  $(n \ne 0)$\\
$[2]:\;(ab^2)^n a$\\  
$[3]:\;(ab^2)^n ab$\\ 
$[4]:\;(ab^2)^n ab^3$\\
$[5]:\;(ab^2)^n ab (ab^2)^m$\\ 
$[6]:\;(ab^2)^n ab^3 (ab^2)^m$\\
$[7]:\;(ab^2)^n ab (ab^2)^m a$\\ 
$[8]:\;(ab^2)^n ab^3 (ab^2)^m a$\\
$[9]:\;(ab^2)^n ab (ab^2)^m ab$\\ 
$[10]:\;(ab^2)^n ab^3 (ab^2)^m ab$\\
$[11]:\;(ab^2)^n ab (ab^2)^m ab^3$\\
$[12]:\;(ab^2)^n ab^3 (ab^2)^m ab^3$\\\\
We start from $[1]$:\\
$ab^2=(b^2,a,b^2,a)\cdot(12)(34)(e,e,a,a)=(12)(34)(b^2,a,b^2a,e)$\\
$(ab^2)^2=(b^2a, ab^2, b^2a, b^2a)$\\
$(ab^2)^{2k}=((b^2a)^k, (ab^2)^k, (b^2a)^k, (b^2a)^k)$(*)\\
$(ab^2)^{2k+1}=(12)(34)((b^2a)^k b^2, (ab^2)^k a, (b^2a)^{k+1}, (b^2a)^k)$
As odd powers of $ab^2$ are non-trivial (the permutation is (12)(34)), (*) implies that all positive powers of $(ab)^2)$ are non-trivial.\\
$(ab^2)^n a \ne e$ as conjugated with $a$ or $b^2$.
Notice that $a$ has a trivial permutation, and $b$ has a permutation of order 4. Therefore in each trivial $ab$-word the total power of b divides 4. So $[3]-[8]$ are non-trivial (the total power of $b$ is odd). The properties of $[9]-[12]$ depend of $ab^2$ terms' parity. We have to consider such types of words:\\\\
$[9.1]:\; (ab^2)^{2k+1} ab (ab^2)^{2t} ab$\\ 
$[9.2]:\; (ab^2)^{2k} ab (ab^2)^{2t+1} ab$\\
$[10.1]:\; (ab^2)^{2k} ab^3 (ab^2)^{2t} ab$\\ 
$[10.2]:\; (ab^2)^{2k+1} ab^3 (ab^2)^{2t+1} ab$\\
$[11.1]:\; (ab^2)^{2k} ab (ab^2)^{2t} ab^3$\\ 
$[11.2]:\; (ab^2)^{2k+1} ab (ab^2)^{2t+1} ab^3$\\ 
$[12.1]:\; (ab^2)^{2k+1} ab^3 (ab^2)^{2t} ab^3$\\ 
$[12.2]:\; (ab^2)^{2k} ab^3 (ab^2)^{2t+1} ab^3$\\\\
for arbitrary $k \ge 0, t \ge 0$. Here we omit $(ab^2)^{2k} ab (ab^2)^{2t} ab$ and other types of words, in which the total power of $b$ doesn't divide 4. Firstly we find terms $[9.1]-[12.2]$ consist of:\\\\
$(ab^2)^{2k+1} ab =$\\
$= (12)(34)((b^2a)^k b^2, (ab^2)^k a, (b^2a)^{k+1}, (b^2a)^k) \cdot (1324)(b^2,e,b^2,e) =$\\ 
                      $=(1423)((b^2a)^k b^2, (ab^2)^{k+1}, (b^2a)^{k+1}, (b^2a)^k b^2)$\\\\
$(ab^2)^{2k+1} ab^3 =$\\
$= (1423)((b^2a)^k b^2, (ab^2)^{k+1}, (b^2a)^{k+1}, (b^2a)^k b^2) \cdot (12)(34)(e,e,a,a)=$\\
$=(1324)((b^2a)^{k+1}, (ab^2)^{k+1}a, (b^2a)^{k+1}, (b^2a)^k b^2) $\\\\     
$(ab^2)^{2k} ab =$\\
$= ((b^2a)^k, (ab^2)^k, (b^2a)^k, (b^2a)^k)\cdot (1324)(b^2,e,b^2,e) = $\\$=(1324)((b^2a)^k b^2, (ab^2)^k, (b^2a)^k b^2, (b^2a)^k)$\\\\
$(ab^2)^{2k} ab^3 =$\\
$= (1324)((b^2a)^k b^2, (ab^2)^k, (b^2a)^k b^2, (b^2a)^k)\cdot (12)(34)(e,e,a,a) =$\\$= 
                    (1423)((b^2a)^{k+1}, (ab^2)^k a, (b^2a)^k b^2, (b^2a)^k)$\\\\
Now we consider $[9.1]-[12.2]$.\\\\
$[9.1]:$\\
$(ab^2)^{2k+1} ab (ab^2)^{2t} ab =(1423)((b^2a)^k b^2, (ab^2)^{k+1}, (b^2a)^{k+1}, (b^2a)^k b^2) \cdot$\\$\cdot 
                                    (1324)((b^2a)^t b^2, (ab^2)^t, (b^2a)^t b^2, (b^2a)^t) = (!, !, (b^2a)^{k+t+1} b^2, !) \ne e$, because $(b^2a)^{k+t+1} b^2$ is conjugated with $a$ or $b^2$.\\
$[9.2]:$\\
$(ab^2)^{2k} ab (ab^2)^{2t+1} ab = (1324)((b^2a)^k b^2, (ab^2)^k, (b^2a)^k b^2, (b^2a)^k) \cdot$\\$\cdot 
                                    (1423)((b^2a)^t b^2, (ab^2)^{t+1}, (b^2a)^{t+1}, (b^2a)^t b^2) = (!,!,(b^2a)^k b^2 (ab^2)^{t+1},!)=(!,!,(b^2a)^{k+t+1} b^2, !) \ne e.$\\
$[10.1]:$\\
$(ab^2)^{2k} ab^3 (ab^2)^{2t} ab =  (1423)((b^2a)^{k+1}, (ab^2)^k a, (b^2a)^k b^2, (b^2a)^k) \cdot$\\$\cdot (1324)((b^2a)^t b^2, (ab^2)^t, (b^2a)^t b^2, (b^2a)^t) = ((b^2a)^{k+1+t},!,!,!) \ne e$\\
$[10.2]:$\\
$\;(ab^2)^{2k+1} ab^3 (ab^2)^{2t+1} ab = (1324)((b^2a)^{k+1}, (ab^2)^{k+1}a, (b^2a)^{k+1}, (b^2a)^k b^2) \cdot$\\$\cdot
(1423)((b^2a)^t b^2, (ab^2)^{t+1}, (b^2a)^{t+1}, (b^2a)^t b^2) = ((b^2a)^{k+t+2},!,!,!) \ne e$\\                                        $[11.1]$\\
                                         $(ab^2)^{2k} ab (ab^2)^{2t} ab^3 = (1324)((b^2a)^k b^2, (ab^2)^k, (b^2a)^k b^2, (b^2a)^k)\cdot$\\$\cdot 
                                    (1423)((b^2a)^{t+1}, (ab^2)^t a, (b^2a)^t b^2, (b^2a)^t) = (!,!,!,(b^2a)^{k+1+t})\ne e$\\
$[11.2]:$\\
$(ab^2)^{2k+1} ab (ab^2)^{2t+1} ab^3 = (1423)((b^2a)^k b^2, (ab^2)^{k+1}, (b^2a)^{k+1}, (b^2a)^k b^2)\cdot$\\$\cdot 
                                        (1324)((b^2a)^{t+1}, (ab^2)^{t+1}a, (b^2a)^{t+1}, (b^2a)^t b^2) = 
                                        (!,!,!,(b^2a)^k b^2(ab^2)^{t+1}a) = (!,!,!,(b^2a)^{k+t+2}) \ne e$\\
$[12.1]:$\\
$(ab^2)^{2k+1} ab^3 (ab^2)^{2t} ab^3 = (1324)((b^2a)^{k+1}, (ab^2)^{k+1}a, (b^2a)^{k+1}, (b^2a)^k b^2)\cdot$\\$\cdot
                                        (1423)((b^2a)^{t+1}, (ab^2)^t a, (b^2a)^t b^2, (b^2a)^t) = ((b^2a)^{k+t+1} b^2, !,!,!) \ne e$\\ 
$[12.2]:$\\
$(ab^2)^{2k} ab^3 (ab^2)^{2t+1} ab^3 = (1423)((b^2a)^{k+1}, (ab^2)^k a, (b^2a)^k b^2, (b^2a)^k) \cdot$\\$\cdot 
                                        (1324)((b^2a)^{t+1}, (ab^2)^{t+1}a, (b^2a)^{t+1}, (b^2a)^t b^2) = ((b^2a)^{k+t+1} b^2, !,!,!)\ne e$\\ 
So, the automaton actually defines $G_{AB}$.           
\end{proof}                                        
         
\newpage

\section{Automaton, defining a direct product $G\times G$}
	
Let X be an alphabet, $|X|=d$. For arbitrary words $a,b \in X^w$, $a=a_1...a_n...$, $b=b_1...b_n...$ define $a\times b$ as word from $X^w$ obtained by mixing letters of $a$ and $b$: $$a \times b := a_1b_1...a_nb_n...$$ 

Let A be an automaton with states $q_1, ..., q_n$, defining $G:=G_A$, and it's representation using wreath recursion is as follows:
\begin{center} 
 $q_1=\pi_1(q_{N(1,1)}, ..., q_{N(1,d)})$\\
 ...\\
 $q_n=\pi_n(q_{N(n,1)}, ..., q_{N(n,d)})$
\end{center}  
Here $\pi_i$ are permutations of $X$, $N(i,j)\in\{1,...n\}$, $1 \le i \le n$,  $1 \le j \le d$.  
Define a new automaton $B$ with states $q_1^1, ..., q_n^1,\:q_1^2, ..., q_n^2$ such that:
\begin{center} 
 $q_i^1=\pi_i(q_{N(i,1)}^2, ..., q_{N(i,d)}^2)$\\
 $q_i^2=(q_{N(i,1)}^1, ..., q_{N(i,d)}^1)$
\end{center} 
It's easy to see that
\begin{center}	
  $q_i^1(a \times b)=q_i(a)\times b\;\;\;(1)$\\
  $q_i^2(a \times b)=a\times q_i(b)\;\;\;(2)$
\end{center}	  
So $q_i^1$ and $q_j^2$ commute for each $i,j\in\{1,...,n\}$. It implies that both $G_1:=<q_1^1, ..., q_n^1>$ and $G_2:=<q_1^2, ..., q_n^2>$ are the normal subgroups of $G_B$. According to $(1)$ and $(2)$, they are isomorphic to $G$. Actually, the isomorphisms can be defined $\phi(q_i^1)=q_i$, $\psi(q_i^2)=q_i$. As $<G_1, G_2>=G_B$ and $G_1\cap G_2 ={e}$ (elements of $G_1$ change only odd letters, and of $G_2$ - only even), we have: $G_B \simeq G\times G$, so B defines the group $G\times G$.

\begin{remark} 
	Similarly all positive integer direct powers of $G$ can be obtained. For power $L\in\mathbb{N}$ one can consider the automaton with states $\{q_i^j: \; 1\le i \le n, \; 1\le j \le L\}$, defined as follows:\\\\
     $q_i^1=\pi_i(q_{N(i,1)}^2, ..., q_{N(i,d)}^2)$\\
     $q_i^j=(q_{N(i,1)}^{j+1}, ..., q_{N(i,d)}^{j+1})$, $2\le j \le L-1$\\
     $q_i^L=(q_{N(i,1)}^{1}, ..., q_{N(i,d)}^{1})$\\
    
Then $q_i^j$ changes only letters with position numbers equal to $j$ modulo $L$.
\end{remark}


\newpage
\begin{thebibliography}{4}
	\bibitem{}
	R. I. Grigorchuk, V. V. Nekrashevich, V. I. Sushchanskii. Automata,
	dynamical systems, and groups. Grigorchuk R. I. (ed.), Dynamical
	systems, automata, and infinite groups. Proc. Steklov Inst. Math. 231
	(2000), 128-203.
	\bibitem{}  R. I. Grigorchuk. On a question of Wiegold and torsion images of Coxeter groups. Journal 'Algebra and Discrete Mathematics', Number 4 (2009), 78-96.
	\bibitem{}  V. Nekrashevych. Self-similar groups. Math. Surveys and Monographs
	117. Amer. Math. Soc., Providence, RI, 2005.
	\bibitem{} I. Bondarenko. Groups Generated by Bounded Automata and Their Schreier Graphs. PhD Dissertation (Texas $A\&M$ Univ., College Station, TX, 2007).
	
\end{thebibliography}
\end{document}